\documentclass{mcom-l}

\usepackage{url}
\newtheorem{theorem}{Theorem}[section]

\newtheorem{proposition}[theorem]{Proposition}
\newtheorem{corollary}[theorem]{Corollary}
\theoremstyle{definition}

\theoremstyle{remark}
\newtheorem{remark}[theorem]{Remark}

\numberwithin{equation}{section}

\DeclareMathOperator{\Aut}{Aut}
\DeclareMathOperator{\Hol}{Hol}
\DeclareMathOperator{\Norm}{N}

\newcommand\numberbraces{1\,515\,429}

\begin{document}

\title
{Enumeration of left braces with additive group $C_4\times C_4\times C_4$}


\author{A. Ballester-Bolinches}
\address{Departament de Matem\`atiques, Universitat de Val\`encia, Dr.\ Moliner, 50, 46100 Burjassot, Val\`encia, Spain}
\email{Adolfo.Ballester@uv.es}

\author{R. Esteban-Romero}
\address{Departament de Matem\`atiques, Universitat de Val\`encia, Dr.\ Moliner, 50, 46100 Burjassot, Val\`encia, Spain}
\email{Ramon.Esteban@uv.es}

\author{V. P\'erez-Calabuig}
\address{Departament de Matem\`atiques, Universitat de Val\`encia, Dr.\ Moliner, 50, 46100 Burjassot, Val\`encia, Spain}
\curraddr{}
\email{Vicent.Perez-Calabuig@uv.es}
\thanks{}

\date{}

\dedicatory{}

\begin{abstract}
    We show that the number of isomorphism classes of left braces of order~$64$ with additive group isomorphic to $C_4\times C_4\times C_4$ is $\numberbraces$.

\end{abstract}

\maketitle

\section{Introduction}

The notion of \emph{skew left brace} was introduced by Guarnieri and Vendramin in \cite{GuarnieriVendramin17}. A \emph{skew left brace} is a triple $(B, {+}, {\cdot})$ where $B$ is a set and ${+}$, ${\cdot}$ are two binary operations on~$B$ such that $(B, {+})$ and $(B, {\cdot})$ are groups and that are related by the distributive-like law $a(b+c)=ab-a+ac$. As it is common in the theory of skew left braces, we omit the sign $\cdot$ and we write $ab$ instead of $a\cdot b$ and we use $-a$ to denote the inverse of $a$ in the group $(B, {+})$; the expression $a-b$ means $a+(-b)$. When, in addition,
$(B, {+})$ is an abelian group, then we speak of a \emph{left brace}, a notion introduced by Rump in his seminal paper \cite{Rump07}.

One of the most natural problems in the theory of skew left braces is the determination of skew left braces of a given finite order. Guarnieri and Vendramin present in \cite[Algorithm~5.1]{GuarnieriVendramin17} an algorithm to enumerate all skew left braces with a given finite additive group~$A$. They presented in \cite{GuarnieriVendramin17} the numbers of isomorphism classes of left braces of order~$n$ for $n\le 120$ except for $n\in\{32,64, 81, 96\}$ obtained with their implementation of this algorithm in \textsf{Magma} \cite{BosmaCannonFiekerSteel16-Magma}. In fact, they claim in their paper: ``With current computational resources, we were not able to compute the number of non-isomorphic left braces of orders~$32$, $64$, $81$ and~$96$.'' This computation appears open as \cite[Problem~6.1]{GuarnieriVendramin17}. Vendramin posed in \cite[Problem~2.13]{Vendramin19-agta} the problem of constructing all left braces of order~$32$, for which he presented some partial results on \cite[Table~2.3]{Vendramin19-agta}. He also wrote as a comment to this problem: ``The number of (skew) left braces of size~$64$, $96$ or $128$ seems to be extremely large and our computational methods are not strong enough to construct them all.''

Bardakov, Neshchadim, and Yadav presented in \cite[Algorithm~2.4]{BardakovNeshchadimYadav20} a modification of \cite[Algorithm~5.1]{GuarnieriVendramin17} for the computation of finite skew left braces with a given additive group. They were able to enumerate the isomorphisms classes of skew left braces of orders~$32$ and~$81$, as well as the left braces of order~$96$. With respect to the isomorphism classes of left braces of order~$64$, they were able to enumerate them in \cite[Table~6]{BardakovNeshchadimYadav20} for all isomorphism classes of abelian groups of order~$64$, except for the cases of additive group isomorphic to $C_4\times C_4\times C_4$ (\texttt{SmallGroup(64, 55)} in the notation of \textsf{GAP} \cite{GAP4-11-1}) and to $C_2\times C_2\times C_4\times C_4$ (\texttt{SmallGroup(64, 192)}).

In this paper we obtain the isomorphism classes of left braces with additive group isomorphic to $C_4\times C_4\times C_4$.
\begin{theorem}
  There are $\numberbraces$ isomorphism classes of left braces of order~$64$ whose additive group is isomorphic to $C_4\times C_4\times C_4$.
\end{theorem}

Table~\ref{tab-mult} summarises these isomorphism classes by the isomorphism class of the multiplicative group. For instance, the first two entries of the first row of the table ($4$, $42$) mean that there are two isomorphism classes of braces with additive group $C_4\times C_4\times C_4$ and multiplicative group isomorphic to \texttt{SmallGroup(64, 4)}.
\begin{table}[htbp]
  \caption{Number of isomorphism classes of left braces with additive group $C_4\times C_4\times C_4$ by multiplicative group}\label{tab-mult}\medskip
  \footnotesize
  \centering
  \begin{tabular}{|rr|rr|rr|rr|rr|}\hline
    id&\#&id&\#&id&\#&id&\#&id&\#\\\hline
4  & 42    & 75  & 43\,525 & 119 & 67  & 165 & 81    & 221 & 26\,932 \\
5  & 40    & 76  & 14\,422 & 120 & 66  & 166 & 130   & 222 & 5\,404  \\
6  & 4     & 77  & 28\,837 & 121 & 67  & 167 & 1     & 223 & 32\,308 \\
7  & 22    & 78  & 43\,304 & 122 & 66  & 168 & 47    & 224 & 2\,763  \\
8  & 74    & 79  & 43\,180 & 123 & 1   & 169 & 46    & 225 & 2\,731  \\
9  & 64    & 80  & 14\,420 & 124 & 40  & 170 & 47    & 226 & 21\,671 \\
10 & 16    & 81  & 43\,123 & 125 & 40  & 172 & 46    & 227 & 64\,744 \\
13 & 2     & 82  & 4\,118  & 128 & 63  & 173 & 34    & 228 & 21\,644 \\
14 & 16    & 83  & 32    & 129 & 135 & 175 & 34    & 229 & 14\,412 \\
17 & 50    & 84  & 33    & 130 & 163 & 176 & 48    & 230 & 3\,618  \\
18 & 82    & 85  & 42    & 131 & 92  & 177 & 34    & 231 & 10\,802 \\
19 & 12      & 86  & 32    & 132 & 164 & 178 & 82      & 232 & 43\,179 \\
20 & 156     & 87  & 102   & 133 & 224 & 179 & 1       & 233 & 21\,576 \\
21 & 36      & 88  & 93    & 134 & 287 & 181 & 1       & 234 & 43\,092 \\
22 & 24      & 89  & 112   & 135 & 98  & 182 & 1       & 235 & 10\,786 \\
23 & 406     & 90  & 947   & 136 & 326 & 192 & 976     & 236 & 10\,794 \\
24 & 48      & 91  & 266   & 137 & 130 & 193 & 3\,453  & 237 & 10\,790 \\
25 & 106     & 92  & 149   & 138 & 485 & 194 & 2\,826  & 240 & 10\,828 \\
32 & 294     & 93  & 58    & 139 & 515 & 195 & 6\,945  & 241 & 21\,686 \\
33 & 283     & 94  & 164   & 140 & 1   & 196 & 16\,440 & 242 & 3\,664  \\
34 & 133     & 95  & 57    & 141 & 42  & 197 & 3\,624  & 243 & 21\,516 \\
35 & 160     & 96  & 101   & 142 & 73  & 198 & 5\,468  & 244 & 10\,820 \\
36 & 2       & 97  & 109   & 143 & 114 & 199 & 6\,987  & 246 & 13    \\
37 & 10      & 98  & 135   & 144 & 73  & 200 & 489     & 247 & 55    \\
55 & 567     & 99  & 65    & 145 & 138 & 201 & 5\,430  & 248 & 56    \\
56 & 3\,757  & 100 & 105   & 146 & 131 & 202 & 7\,632  & 249 & 66    \\
57 & 3\,640  & 101 & 316   & 147 & 35  & 203 & 19\,248 & 250 & 12    \\
58 & 21\,838 & 102 & 280   & 148 & 96  & 204 & 13\,610 & 251 & 43    \\
59 & 21\,628 & 103 & 39    & 149 & 131 & 205 & 16\,629 & 252 & 31    \\
60 & 3\,908  & 104 & 42    & 150 & 35  & 206 & 27\,099 & 253 & 120   \\
61 & 21\,812 & 105 & 32    & 151 & 96  & 207 & 8\,277  & 254 & 147   \\
62 & 11\,052 & 106 & 17    & 152 & 64  & 208 & 5\,407  & 255 & 203   \\
63 & 10\,850 & 107 & 17    & 153 & 16  & 209 & 9\,052  & 256 & 246   \\
64 & 7\,193  & 108 & 21    & 154 & 48  & 210 & 32\,312 & 257 & 52    \\
65 & 3\,682  & 109 & 67    & 155 & 2   & 211 & 1\,919  & 258 & 144   \\
66 & 43\,617 & 110 & 24    & 156 & 43  & 212 & 2\,763  & 259 & 92    \\
67 & 43\,927 & 111 & 24    & 157 & 2   & 213 & 13\,608 & 260 & 193   \\
68 & 86\,219 & 112 & 52    & 158 & 43  & 214 & 2\,758  & 261 & 1\,011  \\
69 & 87\,259 & 113 & 54    & 159 & 1   & 215 & 19\,242 & 262 & 173   \\
70 & 43\,183 & 114 & 46    & 160 & 66  & 216 & 19\,092 & 263 & 2\,052  \\
71 & 21\,837 & 115 & 68    & 161 & 52  & 217 & 8\,253  & 264 & 1\,921  \\
72 & 21\,725 & 116 & 140   & 162 & 52  & 218 & 13\,543 & 265 & 489   \\
73 & 14\,585 & 117 & 68    & 163 & 69  & 219 & 53\,836 & 266 & 503   \\
74 & 14\,420 & 118 & 1     & 164 & 81  & 220 & 32\,333 & 267 & 10    \\\hline 
\end{tabular}
\end{table}

The study of the left braces with additive group isomorphic to $C_2\times C_2\times C_4\times C_4$ has been the object of a later research in \cite{BallesterEstebanPerezC23-braces64-192}, after the submission of this paper. That article has been accepted for publication while writing the revised version for this paper. We must indicate that the techniques used in this paper to classify the left braces with additive group $C_4\times C_4\times C_4$ up to isomorphism are not enough to solve the corresponding problem for the additive group $C_2\times C_2\times C_4\times C_4$ due mainly to the high number of intermediate subgroups needed in the algorithm and we have had to use new ideas that are described with detail in~\cite{BallesterEstebanPerezC23-braces64-192}. The main results of that paper are the following ones.

\begin{theorem}[{\cite[Theorem~1.1]{BallesterEstebanPerezC23-braces64-192}}]
  The number of isomorphism classes of left braces of order~$64$ with additive group isomorphic to $C_2\times C_2\times C_4\times C_4$ is $10\,326\,821$.
\end{theorem}
\begin{corollary}[{\cite[Corollary~1.2]{BallesterEstebanPerezC23-braces64-192}}]
  The number of isomorphism classes of left braces of order~$64$ is $15\,095\,601$.
\end{corollary}

\section{The holomorph of a group}
Most of the results in this section can be considered as folklore. We present them here for the sake of completeness.

Let $(G, {+})$ be a group. The \emph{holomorph} of $G$ is
\[\Hol(G, {+})=\{(g, \alpha)\mid g\in G, \alpha\in\Aut(G)\}\]
with the operation given by
\[(g,\alpha)(h,\beta)=(g+\alpha(h), \alpha\circ\beta).\]
The identity element of $\Hol(G, {+})$ is $(0,1)$ and the inverse of the element $(g,\alpha)\in\Hol(G, {+})$ is \[(g, \alpha)^{-1}=(-\alpha^{-1}(g), \alpha^{-1}).\]
The subgroup $A=\{(0,\alpha)\mid \alpha\in\Aut(G, {+})\}$ is isomorphic to $\Aut(G, {+})$. We identify $A$ with $\Aut(G, {+})$.

The group $\Hol(G, {+})$ acts on $(G,{+})$ by means of
\[(g,\alpha)*h=g+\alpha(h),\qquad g, h\in G.\]
Note that when $(G, {+})$ is the additive group of a vector space, then $\Hol(G, {+})$ can be identified with the affine group on the vector space~$G$ and this corresponds to the natural action of the affine group on~$G$.

Furthermore, let us show that this action is faithful: If $(g, \alpha)*h=(k, \beta)*h$ for all $h\in G$, then $g+\alpha(h)=k+\beta(h)$ for all $h\in G$. In particular, taking $h=0$, we obtain that $g=k$. Consequently, $\alpha(h)=\beta(h)$ for all $h\in G$ and so $\alpha=\beta$. Hence, this action is faithful. Therefore, we can identify $\Hol(G, {+})$ with a subgroup of the group $\Sigma_G$ of all permutations of the set~$G$.

Given $g\in G$, let $\tau_g\colon G\longrightarrow G$ be given by $\tau_g(h)=g+h$, the \emph{left translation} of $G$ defined by $g$. It is clear that $T=\{\tau_g\mid g\in G\}$ is a subgroup of $\Sigma_G$ isomorphic to $G$. Furthermore, we can identify $\tau_g$ with $(g, 1)\in \Hol(G,{+})$. The proof of the following proposition can be found, for instance, in \cite[Application~1.5.2]{Kerber99}.
\begin{proposition}
  The normaliser in $\Sigma_G$ of $T$ coincides with the holomorph of~$(G,{+})$.
\end{proposition}



We now present some characterisations of regular subgroups that are useful for our purposes.

First of all, note that a subgroup $H$ of $\Hol(G, {+})\le \Sigma_G$ is regular if, and only if, it acts transitively on $G$ and the stabiliser $H_g$ of each element $g\in G$ is trivial. In other words, given $h$, $k\in G$, there exists a unique $(g, \alpha)\in H$ such that $(g, \alpha)*k=h$. Suppose that $H$ is a regular subgroup of $\Hol(G, {+})$. Given $h\in G$, there exists a unique $(g, \alpha)\in H$ such that $(g, \alpha)*0=h$. Consequently, $g+\alpha(0)=g=h$. Suppose that $(g, \alpha)$, $(g, \beta)\in H$. Since $g=(g,\alpha)*0=(g, \beta)*0=g$, the regularity of the action of $H$ on $G$ shows that $\alpha=\beta$. Hence, $\alpha$ is uniquely determined by $g\in G$, let us call $\alpha=\lambda_g$. Consequently $H=\{(g, \lambda_g)\mid g\in G\}$, where the map $\lambda\colon G\longrightarrow G$, $\lambda(g)=\lambda_g$ depends on $H$. Furthermore, since the action is transitive, the ``projection'' of $H$ on the $G$-component of $\Hol(G,{+})$ is surjective.

Note that $(0,1)*0=0$, therefore $\lambda_0=1$ and so $H\cap \Aut(G,{+})=\{(0,1)\}$.

\begin{remark}
Other authors have considered the regularity of a subgroup $H$ of $\Hol(G)$ in the following equivalent way: given $k\in G$, there exists a unique $(h, \beta)\in H$ such that $(h, \beta)*k=0$, that is, for every $k\in G$ there exists a unique $(h, \beta)\in H$ such that $h+\beta(k)=0$. We have preferred the opposite point of view because we obtain more easily the expression $H=\{(g, \lambda_g)\mid g\in G\}$ for a regular subgroup $H$ of $\Hol(G)$ (cf.~\cite[Lemma~4.1]{GuarnieriVendramin17}).
\end{remark}

In the following proposition, we assume that $(G, {+})$ is a finite group.

\begin{proposition}\label{prop-count}
  Let $(G, {+})$ be a finite group and let $H$ be a subgroup of $\Hol(G, {+})$. Let us denote by $\pi_G$ the ``projection'' of $\Hol(G, {+})$ on $G$. Then $\lvert H\rvert=\lvert \pi_G(H)\rvert \lvert H\cap \Aut(G, {+})\rvert$.
\end{proposition}
\begin{proof}
  The orbit of $0\in G$ with respect to the action of $\Hol(G)$ on $G$ is $\{(g,\alpha)*0\mid (g,\alpha)\in H\}=\pi_G(H)$ and the stabiliser of $0$ is $\{(g,\alpha)\in H\mid (g,\alpha)*0=0\}=\{(g,\alpha)\in H\mid g=0\}=\Aut(G)\cap H$. The result follows as an application of the orbit-stabiliser theorem.
\end{proof}

Proposition~\ref{prop-count} is useful in the finite case to discard subgroups whose subgroups cannot contain regular subgroups because the ``projection'' to $G$ is not surjective in the computation of all regular subgroups of the holomorph of a given finite group $(G, {+})$. We can do it by means of the following result.
\begin{proposition}\label{prop-surjective}
  Let $(G, {+})$ be a finite group and let $H$ be a subgroup of $\Hol(G, {+})$. The restriction of the ``projection'' $\pi_G$ of $\Hol(G, {+})$ to $H$ is surjective if, and only if, $\lvert H\rvert=\lvert G\rvert \lvert H\cap \Aut(G, {+})\lvert$.
\end{proposition}

Another consequence of Proposition~\ref{prop-count} is the following characterisation of regular subgroups of $\Hol(G, {+})$ for  a finite group $(G, {+})$.
\begin{proposition}[cf.\ \cite{AcriBonatto20-algcoll-ab}]
  Let $(G, {+})$ be a finite group.
  Every two of the following three statements about a subgroup $H$ of $\Hol(G, {+})$ imply the other one.
  \begin{enumerate}
  \item $\lvert H\rvert = \lvert G\rvert$.
  \item $H\cap \Aut(G, {+})=\{(0,1)\}$.
  \item The restriction to $H$ of the ``projection'' $\pi_G$ of $\Hol(G, {+})$ on $G$ is surjective.
  \end{enumerate}
  Moreover, a subgroup $H$ of $\Hol(G,{+})$  satisfying two of the three previous properties (and so the other one) is regular.
\end{proposition}
\begin{proof}


  The fact that every two of the three statements imply the other one is an immediate consequence of Proposition~\ref{prop-surjective}. As in the proof of Proposition~\ref{prop-count}, $\pi_G(H)$ is the orbit of $0$ under the action of $H$ and $H\cap \Aut(G)$ is the stabiliser of $0$. If $H$ satisfies all these properties, then the orbit of $0$ is $G$ and its stabiliser is trivial, that is, $H$ is regular.
\end{proof}

\begin{proposition}\label{prop-prod-sum}
  Let $H$ be a regular subgroup of $\Hol(G)$, say $H=\{(g, \lambda_g)\mid g\in G\}$. Then, given $g$, $k\in G$, $\lambda_{g+\lambda_g(k)}=\lambda_g\circ \lambda_k$ and $\lambda_g^{-1}=\lambda_{\lambda_g^{-1}(-g)}$
\end{proposition}
\begin{proof}
  Note that $(g, \lambda_g)(k, \lambda_k)=(g+\lambda_g(k),\lambda_g\circ\lambda_k)=(g+\lambda_g(k), \lambda_{g+\lambda_g(k)})\in H$. When we apply this to $k=\lambda_g^{-1}(-g)$, we obtain that $1=\lambda_0=\lambda_g\circ\lambda_{\lambda_g^{-1}(-g)}$. The result follows.
\end{proof}

The following fact is mentioned in \cite{BardakovNeshchadimYadav20} and used to improve the algorithms to obtain all skew left braces with a given additive group. 

\begin{proposition}\label{prop-conj-hol-aut}
  Let $H$ be a regular subgroup of $\Hol(G, {+})$ with $(G, {+})$ a group and let $(g, \alpha)\in \Hol(G, {+})$. Then $(g, \alpha)H(g, \alpha)^{-1}$ is again a regular subgroup of $\Hol(G, {+})$. Furthermore, there exists $\beta\in\Aut(G, {+})$ such that $(g, \alpha)H(g,\alpha)^{-1}=(0, \beta)H(0,\beta)^{-1}$.
\end{proposition}

\section{Regular subgroups and skew left braces}\label{sec-regular}

We present now the result that allows us to construct all skew left braces with a given additive group $(G, {+})$. All these results are well known (see, for instance, \cite[Section~4]{GuarnieriVendramin17}) and we present them here for completeness.

\begin{proposition}
  Let $(B, {+}, {\cdot})$ be a skew left brace. Given $a\in B$, let $\lambda_a\colon B\longrightarrow B$ be the lambda map given by $\lambda_a(b)=-a+ab$. Then $H=\{(a, \lambda_a)\mid a\in B\}$ is a regular subgroup of $\Hol(B, {+})$.
\end{proposition}
\begin{proof}
  It is well known that $\lambda_a\in \Aut(B, {+})$ for all $a\in A$. Since
  \[(a, \lambda_a)(b, \lambda_b)=(a+\lambda_a(b), \lambda_a\lambda_b)=(a+\lambda_a(b), \lambda_{a+\lambda_a(b)})\in H\]
  by Proposition~\ref{prop-prod-sum} and, by the same result, $\lambda_a^{-1}=\lambda_{\lambda_a^{-1}(-a)}$ and so $(a, \lambda_a)^{-1}=(\lambda_a^{-1}(-a), \lambda_a^{-1})=(\lambda_a^{-1}(-a),\lambda_{\lambda_a^{-1}(-a)})\in H$, $H$ is a subgroup of $\Hol(B, {+})$. Given $a\in G$, there exists a unique element $(g,\alpha)$ in $H$ such that $(g,\alpha)*0=a$, namely $(g, \alpha)=(a, \lambda_a)$. Therefore, $H$ is regular.
\end{proof}

The proof of the  following proposition can be found in \cite[Theorem~4.2]{GuarnieriVendramin17}.
\begin{proposition}
  Given a regular subgroup $H$ of $\Hol(G, {+})$, with $(G, {+})$ a group, then $H$ admits a structure of skew left brace whose additive group is isomorphic to $(G, {+})$.
\end{proposition}

The following result combines Lemma~2.1 and Theorem~2.2 of \cite{BardakovNeshchadimYadav20}.
\begin{proposition}\label{prop-reg-conj-brace-iso}
  Two regular subgroups $H_1$ and $H_2$ of $\Hol(G, {+})$ induce isomorphic skew left braces if, and only if, they are conjugate by an element of $\Aut(G, {+})$.
\end{proposition}
Note that, by Proposition~\ref{prop-conj-hol-aut}, the condition of Proposition~\ref{prop-reg-conj-brace-iso} can be replaced by conjugation in $\Hol(G, {+})$.

\section{Computational challenges}

By Section~\ref{sec-regular}, the problem of determining the skew left braces with additive group isomorphic to $G=C_4\times C_4\times C_4$ up to isomorphism is reduced to determining the conjugacy classes of regular subgroups of $\Hol(G)$. Furthermore, since $G$ is a $2$-group, every regular subgroup $H$ of $\Hol(G)$ has order $\lvert H\rvert=\lvert G\rvert=64$ and so a conjugate of $H$ is contained in a fixed Sylow $2$-subgroup of~$\Hol(G)$. Hence it is enough to determine all regular subgroups of a Sylow $2$-subgroup of $\Hol(G)$.

Hulpke  \cite{Hulpke99} has developed an algorithm to determine the lattice of subgroups of a finite soluble group~$S$. This algorithm is implemented by means of the function \texttt{SubgroupsSolvableGroup} of \textsf{GAP} \cite{GAP4-11-1}. We summarise this algorithm as follows:

\begin{enumerate}
\item We compute a normal series $S\trianglerighteq  N_1\trianglerighteq \dots \trianglerighteq N_r=1$ with elementary abelian factors.
  
\item We construct by induction the subgroups of $S/N_{i+1}$ from the subgroups of $S/N_i$. Without loss of generality, we assume that $N_{i+1}=1$, $N=N_i$, and we know the subgroups of $S/N$. We have the following possibilities for a subgroup $U$ of $S$:
  \begin{enumerate}
  \item $U$ contains $N$ and thus $U$ is the full preimage of a subgroup of $S/N$ under the natural epimorphism;
  \item $U$ is contained in $N$ and so $U$ is a subspace of the vector space $N$, or
  \item $B:=U\cap N$ is a proper subgroup of~$N$ and $A=NU$ is a subgroup of $G$ that contains properly~$N$.
  \end{enumerate}
  The subgroups of the first type are simply the preimages of the subgroups of $S/N$, that have been computed by induction. The subgroups of the second type are the subspaces of the vector space~$N$. Hence it is enough to consider the third case. In the third case, $B\trianglelefteq U$ and  $B\trianglelefteq N$. Therefore $B\trianglelefteq NU=A$ and $A\le \Norm_S(A)\cap \Norm_S(B)$. It follows that $U/B$ can be computed as a complement of $N/B$ in~$A/B$. 
\end{enumerate}

We cannot apply this algorithm directly to $S=\Hol(G)$ since $\Hol(G)$ is not soluble, but we can apply it to a Sylow $2$-subgroup of $\Hol(G)$. The implementation of \texttt{SubgroupsSolvableGroup} in \textsf{GAP} includes the possibility of adding restrictions like \texttt{ExactSizeConsiderFunction} to avoid the computation of subgroups that do not lead to subgroups of the specified order. This is useful since regular subgroups of $\Hol(G)$ have the same order as $G$. We also note that regular subgroups of $\Hol(G)$ must have a surjective ``projection'' onto $G$ and so we can add the restriction of Proposition~\ref{prop-surjective} to discard all subgroups leading only to non-regular subgroups.

In the \textsf{GAP} implementation of \texttt{SubgroupsSolvableGroup}, the list of all conjugacy classes of subgroups of $G/N_i$ (layer~$i$) and all computed conjugacy classes of $G/N_{i+1}$ (layer $i+1$) are stored at each layer. However, only one group of the layer~$i$ is needed at each step and the groups obtained for the layer $i+1$ will not be used until advancing to the next layer. Furthermore, they can be a lot of subgroups and they can use a large amount of memory, that could eventually exhaust the physical memory. Our approach is to replace saving these subgroups to the memory by saving them to a hard disk at the obvious drawback of speed. Furthermore, in the event of a power loss, we could restart the computation at the exact point it was stopped. We have also modified the algorithm to use space on a hard disk instead of the RAM.

We note that the \textsf{GAP} function \texttt{SubgroupsSolvableGroup}, when it is applied with some restrictions like \texttt{ExactSizeConsiderFunction}, returns all subgroups satisfying these restrictions, but it might return other subgroups. Our implementation includes a final check to remove these eventually extra subgroups.

\section{Our computations}

The implementation of this modified algorithm for the computation of the regular subgroups of a Sylow $2$-subgroup of $\Hol(C_4\times C_4\times C_4)$ produced $31\,367\,678$ conjugacy classes. Of course, some of these classes can have representatives that are not conjugate in the Sylow $2$-subgroup, but can be conjugate in $\Hol(C_4\times C_4\times C_4)$. Consequently, the next natural step is to classify their representatives by conjugation in $\Hol(C_4\times C_4\times C_4)$. Since we were expecting many regular subgroups to be compared by conjugation in $\Hol(C_4\times C_4\times C_4)$, in the last step of the algorithm we classify the regular groups by their isomorphism class, the isomorphism class of the kernel of the action of the brace on the additive group (the set of all elements of the regular subgroup that stabilise all elements of $C_4\times C_4\times C_4$, that coincides with the centraliser of the normal subgroup $C_4\times C_4\times C_4$ in the holomorph as a semidirect product with respect to this action) and the isomorphism class of the quotient by this normal subgroup. We have obtained an overall number of $1\,442$ equivalence classes.

Our idea is to reduce the checking of conjugation to each of these $1\,442$ equivalence classes. The comparisons of different equivalence classes can be performed in parallel by using different processors. Some of these equivalence classes turn out to be small, for example, $74$ of them have only one element and $1\,055$ have at most $100$ elements. In all these classes the comparison by conjugation is fast. However, $63$ equivalence classes have more than $100\,000$ subgroups, $19$ equivalence classes have more than $500\,000$ subgroups, $13$ equivalence classes have more than $800\,000$ subgroups, and the $4$ largest equivalence classes have more than $1\,000\,000$ subgroups. The largest one has $1\,782\,312$ subgroups.

We have decided to refine the $63$ largest equivalence classes by means of the length of the conjugacy class in $\Hol(C_4\times C_4\times C_4)$. This refinement applied to all equivalence classes gives a total number of $2\,353$ equivalence classes. In some cases, this has allowed us to obtain some small numbers of subgroups that can be easily compared by conjugation, but for the largest conjugacy class lengths the computations were still slow. The execution of these comparisons on a computer with an Intel processor i7-11700 that allows the execution of $16$ parallel tasks and $32$~Gb of RAM running GNU/Linux during a couple of months made us guess that we would need more than two years to perform the task.

At that time we applied for the use of the scientific supercomputer \emph{Llu\'\i s Vives} to the Computer Service of the Universitat de València (see \cite{LluisVivesv2}). We thank the Computer Service for granting an immediate access to this machine and for their help installing \textsf{GAP} and solving our doubts. On this machine, we were able to complete the computations in less than two months by running several comparisons in parallel. Our implementation of parallelism in this setting has consisted of running several instances of \textsf{GAP} that select from a list of regular subgroups a unique representative of each conjugacy class in $\Hol(C_4\times C_4\times C_4)$ or select from two lists of regular subgroups for which two elements in the same list are not conjugate in $\Hol(C_4\times C_4\times C_4)$ a list of the subgroups of both lists that contain no pairs of conjugate subgroups. We have not used any particular computer package to run \textsf{GAP} in parallel mode, as our setting has been enough for our purposes. The total number of conjugacy classes of regular subgroups of $\Hol(C_4\times C_4\times C_4)$ we have found is $\numberbraces$. Table~\ref{tab-mult} summarises the numbers of conjugacy classes by the first invariant we consider, the isomorphism class of the multiplicative group of the resulting left brace.

\section*{Acknowledgements}

These results are part of the R+D+i project supported by the Grant 
PGC2018-095140-B-I00, funded by MCIN/AEI/10.13039/501100011033 and by ``ERDF A way of making Europe.'' We thank the anonymous referees for their comments that have allowed us to improve the proofs of some results and the  presentation of the paper. We also thank the Computer Service of the Universitat de Val\`encia for allowing us to use the scientific supercomputer \emph{Llu\'\i s Vives} and for their kind support in setting up the system for the computation and solving our questions.

\section*{Data availability}
The complete list of left braces or order~$64$ is available for use in \textsf{GAP} \cite{GAP4-11-1} on \url{https://github.com/RamonEstebanRomero/braces64} \cite{BallesterEstebanPerezC22-braces64-GitHub}. It also includes the left braces with additive group $C_2\times C_2\times C_4\times C_4$ computed by the authors in~\cite{BallesterEstebanPerezC23-braces64-192}. The storage of these left braces follows the ideas of \cite{BallesterEsteban22} of representing them as triply factorised groups. This repository also includes some \textsf{GAP} functions to use these left braces with the help of the \textsf{YangBaxter} package \cite{VendraminKonovalov22-YangBaxter-0.10.1} for \textsf{GAP}. 



\begin{thebibliography}{10}

\bibitem{AcriBonatto20-algcoll-ab}
E.~Acri and M.~Bonatto, \emph{Skew braces of order {$p^2q$} {I}: Abelian type},
  Alg. Colloq. \textbf{29} (2020), no.~2, 297--320.

\bibitem{BallesterEsteban22}
A.~Ballester-Bolinches and R.~Esteban-Romero, \emph{Triply factorised groups
  and the structure of skew left braces}, Commun. Math. Stat. \textbf{10}
  (2022), 353--370.

\bibitem{BallesterEstebanPerezC22-braces64-GitHub}
A.~Ballester-Bolinches, R.~Esteban-Romero, and V.~P{\'e}rez-Calabuig,
  \emph{Left braces of order {$64$}}, WWW
  \url{https://doi.org/10.5281/zenodo.7406729}, December 2022,
  DOI:10.5281/zenodo.7406729.

\bibitem{BallesterEstebanPerezC23-braces64-192}
\bysame, \emph{Enumeration of left braces with additive group {$C_2\times
  C_2\times C_4\times C_4$}}, Ric. Mat. (2023),
  {\url{https://doi.org/10.1007/s11587-023-00772-2}}.

\bibitem{BardakovNeshchadimYadav20}
V.~G. Bardakov, M.~V. Neshchadim, and M.~K. Yadav, \emph{Computing skew left
  braces of small orders}, Int. J. Algebra Comput. \textbf{30} (2020), no.~4,
  839--851.

\bibitem{BosmaCannonFiekerSteel16-Magma}
W.~Bosma, J.~J. Cannon, C.~Fieker, and A.~Steel, \emph{Handbook of
  {\textsc{magma}} functions}, {2.19} ed., 2016.

\bibitem{LluisVivesv2}
{Computer Service, Universitat de Val{\`e}ncia}, \emph{Lluisvives v2 user's
  guide}, Universitat de Val{\`e}ncia, 2020,
  \url{https://www.uv.es/calcul/Manuales/LluisVives2_user_manual.pdf} (visited
  8th April, 2022).

\bibitem{GAP4-11-1}
The GAP~Group, \emph{{GAP} -- {Groups}, {Algorithms}, and {Programming},
  {Version} {4.11.1}}, 2021.

\bibitem{GuarnieriVendramin17}
L.~Guarnieri and L.~Vendramin, \emph{Skew-braces and the {Y}ang-{B}axter
  equation}, Math.\ Comp. \textbf{86} (2017), no.~307, 2519--2534. \MR{3647970}

\bibitem{Hulpke99}
A.~Hulpke, \emph{Computing subgroups invariant under a set of automorphsms}, J.
  Symbolic Comput. \textbf{27} (1999), no.~4, 415--427.

\bibitem{Kerber99}
A.~Kerber, \emph{Applied finite group actions}, 2nd ed., Algorithms and
  Combinatorics, vol.~19, Springer-Verlag, Berlin, Heidelberg, New York, 1999.

\bibitem{Rump07}
W.~Rump, \emph{Braces, radical rings, and the quantum {Y}ang-{B}axter
  equation}, J. Algebra \textbf{307} (2007), 153--170. \MR{2278047}

\bibitem{Vendramin19-agta}
L.~Vendramin, \emph{Problems on skew left braces}, Adv. Group Theory Appl.
  \textbf{7} (2019), 15--37.

\bibitem{VendraminKonovalov22-YangBaxter-0.10.1}
L.~Vendramin and O.~Konovalov, \emph{Yang{B}axter: combinatorial solutions for
  the {Y}ang-{B}axter equation}, August 2022, version 0.10.1,
  \url{https://gap-packages.github.io/YangBaxter/}.

\end{thebibliography}

\providecommand{\bysame}{\leavevmode\hbox to3em{\hrulefill}\thinspace}
\providecommand{\MR}{\relax\ifhmode\unskip\space\fi MR }
\providecommand{\MRhref}[2]{%
  \href{http://www.ams.org/mathscinet-getitem?mr=#1}{#2}
}
\providecommand{\href}[2]{#2}

\end{document}